\RequirePackage{snapshot}
\documentclass[leqno, 11pt]{amsart}
\usepackage{amssymb,amsmath,amsthm,graphicx,multirow,setspace,url}
\usepackage{amscd}
\input xy
\xyoption{all}

\usepackage{graphicx}

\title{The Monomial Conjecture and order ideals II}
\author{S. P. Dutta}
\address{Department of Mathematics\\
University of Illinois at Urbana-Champaign\\
1409 West Green Street\\
Urbana, Illinois 61801}

\begin{document}

\newcommand{\hookuparrow}{\mathrel{\rotatebox[origin=c]{90}{$\hookrightarrow$}}}
\newcommand{\hookdownarrow}{\mathrel{\rotatebox[origin=c]{-90}{$\hookrightarrow$}}}


\theoremstyle{plain}
\newtheorem*{theorem*}{Theorem}
\newtheorem{theorem}{Theorem}

\newtheorem{innercustomthm}{Theorem}
\newenvironment{customthm}[1]
  {\renewcommand\theinnercustomthm{#1}\innercustomthm}
  {\endinnercustomthm}

%



\newtheorem*{corollary}{Corollary}

\theoremstyle{remark}
\newtheorem*{remark}{Remark}
\newtheorem*{claim}{Claim}

\newtheorem*{namedtheorem}{\theoremname}
\newcommand{\theoremname}{testing}
\newenvironment{named}[1]{\renewcommand{\theoremname}{#1}
  \begin{namedtheorem}}
	{\end{namedtheorem}}

\newcommand{\gr}{\operatorname{grade}}
\newcommand{\Syz}{\operatorname{Syz}}
\newcommand{\syz}[2]{\Syz^{#1}(#2)}
\newcommand{\ring}[2]{{\mathcal{O}_{#1}(#2)}}
\newcommand{\hm}[3]{H_{#1}^{#2}(#3)}
\newcommand{\Tor}{\operatorname{Tor}}
\newcommand{\tor}[3]{\Tor_{#1}^{#2}(#3)}
\newcommand{\Ext}{\operatorname{Ext}}
\newcommand{\ext}[3]{\Ext_{#1}^{#2}(#3)}
\newcommand{\Id}{\operatorname{Id}}
\newcommand{\im}{\operatorname{Im}}
\newcommand{\coker}{\operatorname{coker}}
\newcommand{\grade}{\operatorname{grade}}
\newcommand{\Ht}{\operatorname{height}}
\newcommand{\Hom}{\operatorname{Hom}}
\newcommand{\ul}[1]{\underline{#1}}

\begin{abstract}
Let $I$ be an ideal of height $d$ in a regular local ring $(R,m,k=R/m)$ of dimension $n$ and let $\Omega$ denote the canonical module of $R/I$.  In this paper we first prove the equivalence of the following: the non-vanishing of the edge homomorphism $\eta_d: \ext{R}{n-d}{k,\Omega} \rightarrow \ext{R}{n}{k,R}$, the validity of the order ideal conjecture for regular local rings, and the validity of the monomial conjecture for all local rings.  Next we prove several special cases of the order ideal conjecture/monomial conjecture.
\end{abstract}

\maketitle

\section{Introduction}\label{Intro}

This paper is a sequel to our work in \cite{D7}.  Let us first state the two conjectures that would constitute the central theme of this article.

\begin{named}{Monomial Conjecture}
(M. Hochster) Let $(R, m)$ be a local ring of dimension $n$ and let $x_1$, $x_2$, $\ldots$, $x_n$ be a system of parameters of $R$.  Then, for every $t>0$, $(x_1 x_2 \cdots x_n)^{t} \notin (x_1^{t+1}, \ldots, x_n^{t+1})$.
\end{named}

\begin{named}{Order Ideal Conjecture}
(Due to Evans and Griffith)  Let $(R,m)$ be a local ring.  Let $M$ be a finitely generated module of finite projective dimension over $R$ and let $\syz{i}{M}$ denote the $i$-th syzygy of $M$ in a minimal free resolution $F_\bullet$ of $M$ over $R$.  If $\alpha$ is a minimal generator of $\syz{i}{M}$, then the ideal $\ring{\syz{i}{M}}{\alpha}$, generated by entries of $\alpha$ in $F_{i-1}$, has grade greater than or equal to $i$ for every $i>0$.

We note that the validity of the assertion for $\ring{\syz{i}{M}}{\alpha}$ as defined above is equivalent to the validity of the assertion that the order ideal for $\alpha (=\{f(\alpha) | f \in \Hom_R(\syz{i}{M}, R) \})$ has grade greater than or equal to $i$ for every $i>0$.
\end{named}

Both these conjectures are open in mixed characteristic. A brief account of these and other equivalent conjectures is provided towards the end of this introduction. In this paper we focus our attention on several aspects of the order ideal conjecture over regular local rings and its connection with the monomial conjecture. As a consequence of the order ideal conjecture Bruns and Herzog \cite{BrH} observed the following:

\bigskip

\line(1,0){300}

AMS Subject Classification:
Primary:  13D02, 13D22

Secondary: 13C15, 13D25, 13H05

\textit{Let $(R, m, k=R/m)$ be a regular local ring and let $I$ be an ideal of height $d \ge 0$ with a minimal set of generators $\{x_i\}_{1 \le i \le t}$.  Let $K_{\bullet}(\ul{x},R)$ denote the Koszul complex corresponding to $\{x_i\}$.  Then $\theta_{i}: K_{i}(\ul{x}; R)\otimes k \rightarrow \tor{i}{R}{R/I, k}$ are $0$-maps for $i>d$.}

\smallskip
In one of our main results \cite[Corollary 1.4 ]{D7} we proved that \textit{if, for any $d \ge 0$ and for any almost complete intersection ideal $I=(x_1, \ldots, x_d, x_{d+1})$ of height $d$ in a regular local ring $R$, $\theta_{d+1}: K_{d+1}(\ul{x};R) \otimes k \rightarrow \tor{d+1}{R}{R/I, k}$ is the $0$-map, then the monomial conjecture is valid for all local rings.}

%
%

Equivalently, we showed that \textit{if the canonical module $ \Omega$ of $R/I$ is such that $\syz{d}{\Omega}$ has a free summand for $d > 0$, then the monomial conjecture is valid \cite[Section 1.6]{D7}}. The non-vanishing of $\theta_{d+1}$ for almost complete intersection ideals refers to a very special case of the consequence of the order ideal conjecture. In the light of the significance of this special case mentioned above, the following question emerged: How ``special" is this special case? We are now able to answer this question in the first theorem of this paper.

\begin{customthm} {2.1} \label{Tm2.1}
The following statements are equivalent:
\begin{enumerate}
\item[(1)] The order ideal conjecture is valid over any regular local ring $R$.
\item[(2)] For every ideal $I$ of $R$ of height $d>0$, the $d$-th syzygy in a minimal free resolution of the canonical module of $R/I$ has a free summand.
\item[(3)] The same statement as in $(2)$ for every prime ideal of $R$. 
\item[(4)] For every almost complete intersection ideal $I$ of height $d>0$, with $I=(x_1, \ldots, x_{d+1})$, the map $\theta_{d+1}: K_{d+1}(\ul{x}; R)\otimes k \rightarrow \tor{d+1}{R}{R/I,k}$ is the $0$-map.
\end{enumerate}
\end{customthm}

Thus, the very special case of a consequence of the order ideal conjecture, as mentioned above in (4), implies the order ideal conjecture over regular local rings.

For our next theorem, we need the following set up.  Let $(R, m, k)$ be a regular local ring of dimension $n$ and let $I$ be an ideal of height $d$.  From the associativity property of Hom and tensor product, we obtain two spectral sequences with $E_{2}^{i,j}$ terms $\ext{R}{i}{k, \ext{R}{j}{R/I,R}}$ and $\ext{R}{i}{\tor{j}{R}{k,R/I},R}$ that converge to the same limit.  Since the height of $I=d$, we have $\ext{R}{j}{R/I,R}=0$ for $j<d$.  When $i+j=n$ we obtain the ``edge'' homomorphism $\eta_d: \ext{R}{n-d}{k,\ext{}{d}{R/I,R}}\rightarrow\ext{R}{n}{k,R}$.  We note that in this case, the second spectral sequence collapses to the limit $\ext{R}{n}{k, R}$. The obvious question is : when is $\eta_d \neq 0$? Our next Theorem answers this question by relating it with two of the most important homological conjectures in commutative algebra.

\begin{customthm} {2.2} \label{Tm2.2}
The following statements are equivalent:
\begin{enumerate}
\item[(1)]	For every regular local ring $R$ of dimension $n$ and for any ideal $I$ of height $d$ of $R$, the edge homomorphism $\eta_d: \ext{R}{n-d}{k, \ext{}{d}{R/I,R}}$ $\rightarrow \ext{R}{n}{k,R}$ is non-zero.
\item[(2)] The same statement where $I$ is replaced by a prime ideal $P$ of $R$.
\item[(3)] The order ideal conjecture is valid over regular local rings.
\item[(4)] The monomial conjecture is valid for all local rings.
\end{enumerate}
\end{customthm}


The proof of this theorem takes into account our previous work on the canonical element conjecture in \cite{D3} and \cite{D4}. This theorem places the order ideal conjecture over regular local rings at a central position among several homological conjectures. It also shows that the statement of the order ideal conjecture over local rings presents a much stronger assertion than that of the monomial conjecture over local ones. We refer the reader to \cite{D8} for our work involving the general case of the order ideal conjecture.

A very special case of ``$(1) \rightarrow (4)$" in the above theorem was established in [6] over Gorenstein local rings. As a corollary of Theorem 2.2 we obtain the following result which this author has been seeking since his work in [6].

\begin{customthm} {2.3} \label{Tm2.3}
Let $(R,m,k)$ be an equicharacteristic regular local ring of dimension n. Let I be an ideal of R
of height d. Then the edge homomorphism $\eta_d : \Ext_R^{n-d} (k, \Ext^d (R/I, R)) \to \Ext_R^n (k, R)$ is non-zero.
\end{customthm}

In Section 3 we present proofs of several aspects of the order ideal conjecture (we do not mention the results in \cite{D7} here). In our first theorem of this section we prove the following:

\begin{customthm} {3.1} \label{Tm3.1}
Let $M$ be a finitely generated module over a regular local ring $R$ such that $M$ satisfies Serre-condition $S_1$ and $M$ is equidimensional.  Let $d$ be the codimension of $M$.  Then minimal generators of $\syz{i}{M}$ with $i\ge d$ satisfies the order ideal conjecture.
\end{customthm}

In our next proposition we consider a special generator of $\syz{i}{M}$ and show that when such a generator exists, the restriction of finite projective dimension is not required. We prove the following:

\it Let $M$ be a finitely generated module over a Cohen-Macaulay ring. Let $\alpha$ be a minimal generator of $\syz{i}{M}$, $i \geq 1$, such that it generates a free summand of $\syz{i}{M}$. Then $\gr \ring{\syz{i}{M}}{\alpha} \ge i$.
\rm

Our next theorem points out the connection between $\syz{d} {R/P}$ and $\syz{d}{\Omega}$ in terms of having a free summand.

\begin{customthm}{3.3} \label{Tm3.3}
If $P$ is a prime ideal of height d in a regular local ring R such that $\syz{d} {R/P})$ has a free summand, then $\syz{d}{\Omega}$, where $\Omega$ is the canonical module of $R/P$, also has a free summand.
\end{customthm}

As a corollary we derive that \textit{the direct limit map $\ext{R}{n-d}{k,R/P} \rightarrow \hm{m}{n-d}{R/P}$ is non-null and hence $R/P$ satisfies a stronger version of the monomial conjecture.}

\medskip
In our final theorem of this paper we prove the following:

\begin{customthm} {3.4}
Let $R$ be a regular local ring and let $P$ be a prime ideal of height $d$. Let $\nu(\Omega)$ denote the minimal number of generators the canonical module $\Omega$ of $R/P$. If $\nu(\Omega) = 2$ and $R/P$ satisfies Serre-condition $S_4$, then $\syz{d}{\Omega}$ has a free summand.
\end{customthm}

As a corollary we derive the following: \textit{Let $A$ be a complete local domain satisfying Serre-condition $S_4$. Suppose $\nu(\Omega) =2$.  Then $A$ satisfies the monomial conjecture.}

We would remind the reader that the validity of the monomial conjecture when $\nu(\Omega)=1$ and $A$ satisfies Serre-condition $S_3$ has already been established in \cite{D4}.

\medskip

 Let us now provide a brief history.  In \cite{H1} Hochster proved the monomial conjecture in equicharacteristic and proposed it as a conjecture for local rings in mixed characteristic. He also showed that this conjecture is equivalent to the \textit{direct summand conjecture} which asserts the following: \textit{Let $R$ be a regular local ring and let $i: R \rightarrow A$
be a module finite extension. Then $i$ splits as an $R$-module map.} In the eighties in order to prove their syzygy theorem \cite{EG1, EG2, EG3} Evans and Griffith proved an important aspect of order ideals of syzygies over equicharacteristic local rings. The order ideal conjecture stated in this article originates from their work. For their proof of the syzygy theorem Evans and Griffith \cite{EG1} actually needed a particular case of the above conjecture: $M$ is locally free on the punctured spectrum and $R$ is regular local. They reduced the proof of the above special case to the validity of the \textit{improved new intersection conjecture} (also introduced by them) which asserts the following: \textit{Let $(A, m)$ be a local ring of dimension $n$. Let $F_\bullet$ be a complex of finitely generated free $A$- modules.
\[
F_{\bullet}: 0 \rightarrow F_s \rightarrow F_{s-1} \rightarrow \cdots \rightarrow F_1 \rightarrow F_0 \rightarrow 0,
\]
such that $\ell(H_i(F_{\bullet})) < \infty$ for $i >0$ and $H_0(F_{\bullet})$ has a minimal generator annihilated by a power of the maximal ideal $m$. Then dimension of $A \le s$.} Evans and Griffith proved this conjecture for equicharacteristic local rings by using the existence of big Coehn-Macaulay modules in equicharacteristic due to Hochster \cite{H2}.  In the mid-eighties, Hochster proposed the \textit{canonical element conjecture}. In this conjecture Hochster assigned a canonical element $\eta_{R}$ to every local ring $R$ and asserted that $\eta_{R} \ne 0$. An equivalent statement, due to Lipman, of the canonical element conjecture that would be relevant to our work is the following: \textit{let $(A, m, k=A/m)$ be a local ring of dimension $n$ which is homomorphic image of a Gorenstein ring.  Let $\Omega$ denote the canonical module of $A$. Then the direct limit map
\[
\ext{A}{n}{k,\Omega} \rightarrow H_m^{n}(\Omega)
\]
is non-zero} \cite{H3}.

Hochster proved this conjecture in equicharacteristic \cite{H3} and showed that the canonical element conjecture is equivalent to the monomial conjecture and it implies the improved new intersection conjecture. Later this author proved the reverse implication \cite{D1}. Thus the monomial conjecture implies a special case of the order ideal conjecture: the case when $M$ is locally free on the punctured spectrum of $R$ and $R$ is regular local. Over the years different aspects of the monomial conjecture have been studied, special cases have been proved and new equivalent forms have been introduced (see \cite{Bh, BrH, D1, D2, D3, D4, D5, D6, D7, DG, Go, He, K, O, V}).

\medskip
\textbf{Notation.} Throughout this work ``local'' means noetherian local. For any R-module $N$, $N^*$ stands for $\Hom_R(N, R)$ and given any R-linear map $f:M \to N$ between two R-modules, $f^* : M^* \to N^*$ denote the dual map obtained by applying $\Hom (-, R) \ to \ f : M \to N$. Similar notations would be used for maps between two complexes of R-modules.

\section*{Section 2} \label{s2}

\subsection*{}

\begin{customthm} {2.1} \label{Tm2.1}
Let $(R, m, k)$ be a regular local ring.  The following statements are equivalent:
\begin{enumerate}
\item[(1)] The order ideal conjecture is valid on $R$.
\item[(2)] Given any positive integer $d>0$, for every ideal $I$ of height $d$ in $R$, the $d$-th syzygy in a minimal free resolution of the canonical module of $R/I$ has a free summand.
\item[(3)] For every prime ideal of $P$ of height $d$ in $R$, the $d$-th syzygy in a minimal free resolution of the canonical module of $R/P$ has a free summand.
\item[(4)] For every almost complete intersection ideal $I$ of height $d$, with $I=(x_1, \ldots, x_{d},\lambda)$, $\{x_1, \ldots, x_d\}$ being an $R$-sequence, the natural map :
    $K_{d+1}(x_1, \ldots, x_d,\lambda ; R)\otimes k \rightarrow \tor{d+1}{R}{R/I,k}$ is the $0$-map.
\end{enumerate}
\end{customthm}

\begin{proof}
We will prove that $(1) \Rightarrow (2) \Rightarrow (3) \Rightarrow (1)$ and $(1) \Rightarrow (4) \Rightarrow (3)$.

$(1) \Rightarrow (2)$.  Let
\[
(F_{\bullet}, \phi_{\bullet}): \rightarrow R^{t_{d+1}} \stackrel{\phi_{d+1}}{\rightarrow} R^{t_{d}} \stackrel{\phi_{d}}{\rightarrow} R^{t_{d-1}} \rightarrow \cdots \rightarrow R^{t_{1}} \stackrel{\phi_{1}}{\rightarrow} R \rightarrow R/I \rightarrow 0
\]
be a minimal projective resolution of $R/I$ over $R$.  Let $G_i=\coker(\phi_i^*)$ and let $F_{\bullet}^*$ denote the following truncated part of $\Hom_R (F_\bullet, R)$:
\[
{F_{\bullet}^*}: 0 \rightarrow R \stackrel{\phi_{1}^*}{\rightarrow} R^{t_1^*} \rightarrow \cdots \rightarrow R^{t_d^*} \rightarrow R^{t_{d+1}^*} \rightarrow G_{d+1} \rightarrow 0.
\]
We have $H^i(F_{\bullet}^*)=0$ for $i<d$, $H^d(F_{\bullet}^*)=\Omega$ - the canonical module of $R/I$. Let $j: \Omega \hookrightarrow G_d$ denote the natural inclusion of $\ext{}{d}{R/I,R} \hookrightarrow G_d$.  Let
\[
(L_{\bullet}, \psi_{\bullet}):  \rightarrow R^{\gamma_d} \stackrel{\psi_{d}}{\rightarrow} R^{\gamma_{d-1}} \rightarrow \cdots \rightarrow R^{\gamma_0} \rightarrow \Omega \rightarrow 0
\]
be a minimal free resolution of $\Omega$ and $j_{\bullet}: L_{\bullet} \rightarrow F_{\bullet}^*$ denote a left of $j$.

\begin{claim}
$j_d(R^{\gamma_d})=R$.
\end{claim}

Proof of the Claim.  If possible let $j_d (R^{\gamma_d}) \subset m$. The mapping cone $V_{\bullet}$ of $j_{\bullet}$ is a free resolution of $G_{d+1}$. Since $j_d (R^{\gamma_d}) \subset m$, in the minimal free resolution $U_{\bullet}$ of $G_{d+1}$ extracted from $V_{\bullet}$, the copy of $R$ at the $(d+1)-{th}$ spot of $F_{\bullet}^*$ survives. Since height $I=d$ and the coordinates of $\phi_1^*$ are the minimal generators of $I$, the above statement implies that $\syz{d+1}{G_{d+1}}$ has a minimal generator whose entries generate an ideal of height $d$. This contradicts the order ideal conjecture. Hence $j_d (R^{\gamma_d})= R$.

 We can now choose a basis $e_1, \ldots, e_{\gamma_d}$ of $R^{\gamma_d}$ such that $j_d(e_1)=1$ and $j_d(e_i)=0$ for $i>1$.  Due to the commutativity of diagram in $j_{\bullet} : L_{\bullet} \rightarrow F_{\bullet}^*$,
\[
\begin{CD}
R^{\gamma_{d+1}} @>\psi_{d+1}>> R^{\gamma_d} \\
@VVV @VV j_d V \\
0 @>>> R\\
\end{CD}
\]
it follows that $\psi_{d+1}(R^{\gamma_{d+1}})$ is contained in the free submodule generated by $\{e_i\}_{i \ge 2}$.  Hence $\syz{d}{\Omega}$ has a free summand.

$(2) \Rightarrow (3)$ is obvious.

$(3) \Rightarrow (1)$.  Let $M$ be a finitely generated $R$-module and let $S_i$ denote the $i$-th syzygy of $M$ in a minimal free resolution $(F_{\bullet}, \phi_{\bullet})$ of $M$ over $R$. Since $R$ is an integral domain, the assertion of the order ideal conjecture is valid for minimal generators of $S_1$. So we can assume $d>1$. Let $(F_{\bullet}, \phi_{\bullet})$ be given by
\[
F_{\bullet}: \rightarrow R^{t_{d+1}} \stackrel{\phi_{d+1}}{\rightarrow} R^{t_{d}}  \stackrel{\phi_{d}}{\rightarrow}  R^{t_{d-1}} \rightarrow \cdots \rightarrow R^{t_{1}} \stackrel{\phi_{1}}{\rightarrow} R^{t_{0}} \rightarrow M \rightarrow 0.
\]
Let $e$, $\{e_i\}_{1 \le i \le t_{d}-1}$ be a basis of $R^{t_{d}}$ and let $\phi_d(e)=(a_1, a_2, \ldots, a_{t_{d-1}})$.  We denote by $I$ the ideal generated by $a_1, \ldots, a_{t_{d-1}}$ and let $P$ be a minimal prime ideal of $I$ such that $\Ht(P)=\Ht(I)$. If possible let $\Ht(I) = \Ht(P)< d$.

By induction, since $S_d=\Syz^{d-1}(S_1)$, we can assume the height of $P=(d-1)$.  Let $\Omega(=\ext{}{d-1}{R/P, R})$ denote the canonical module of $R/P$.  Let $(P_{\bullet}, \psi_{\bullet})$ denote a minimal free resolution of $R/P$ and let $F_{\bullet}^*$ denote the truncated part of $\Hom_{R}{}({F_{\bullet}, R})$ up to the d-th spot. Let $\eta_0: R^{t_d^*} \rightarrow R$ be given by $\eta_0(e^*)=1$ and $\eta_0(e_i^*)=0$ for $i>0$. Then im $\eta_0 \circ \phi_d^* \subset P$ and hence $\eta_0$ induces a map $\eta: G_d=\coker \phi_d^* \rightarrow R/P$.  Let $\eta_{\bullet}: F_{\bullet}^* \rightarrow P_{\bullet}$ denote a lift of $\eta$.  We have the following commutative diagram.
\begin{equation}\label{CD1}
\begin{gathered}
\xymatrix@=1em{
F_{\bullet}^*: 0 \ar[r]  & R^{t_0^*} \ar[r]^{\phi_1^*} \ar[d]^{\eta_d} & R^{t_1^*} \ar[r] \ar[d] & \cdots \ar[r] & R^{t_{d-1}^*} \ar[r]^{\phi_{d}^*} \ar[d]^{\eta_1} & R^{t_d^*}\ar[r] \ar[d]^{\eta_0} & G_d \ar[r] \ar[d]^{\eta}& 0 \\
P_{\bullet}:  \ar[r] & R^{s_d} \ar[r]^{\psi_d} & R^{s_{d-1}} \ar[r]^{\psi_{d-1}} & \cdots \ar[r] & R^{s_1} \ar[r]^{\psi_1} & R \ar[r] & R/P \ar[r] & 0
}
\end{gathered}
\end{equation}
Let $\widetilde{G}_{d-1}=\coker \psi_{d-1}^*$; then $\Omega$ is a submodule of $\widetilde{G}_{d-1}$.  Applying $\Hom_{R}{}({-,R})$ to (1), we obtain $\eta_\bullet^* : P_\bullet^* \to F_\bullet$ which induces a map $\eta' : \widetilde G_{d-1} \to S_1 = \Syz^1 (M)$.

Since $R$ is regular local, $\Ht$ $(P)=\grade$ $(P) = (d-1)$ and hence $\Omega=\ext{}{d-1}{R/P,R} \ne 0$. Since $\Omega$ is annihilated by $P$ and $S_1$ is torsion free (submodule of $R^{t_0}$), we have $\eta' (\Omega) =0$. By our assumption, $\syz{d-1}{\Omega}$ has a free summand. Let $(L_{\bullet}, a_{\bullet})$, $L_i = R^{\gamma_i}$, denote a minimal free resolution of $\Omega$. Let $\{\widetilde{e}, \widetilde{e}_i \}$ with $1 \le i \le \gamma_{d-1} -1$ be a basis of $R^{\gamma_{d-1}}$ such that $a_{d-1}(\widetilde{e})$ generates a free summand of $\syz{d-1}{\Omega}$.  Then $a_d(R^{\gamma_d})$ is contained in the free submodule of $R^{\gamma_{d-1}}$ generated by $\{ \widetilde{e}_i\} _{1 \le i \le \gamma_{d-1}}$.
Hence $a_d^*({\widetilde e^*}) =0$, where $\{ {\widetilde e^*}, {\widetilde e_i^*} \}$ denote the corresponding dual basis of $R^{r^*_{d-1}}$. We now want to make the following claim.

{\it Claim.} im ${\widetilde e^*} \epsilon \Ext^{d-1} (\Omega, R) = \Hom_R(\Omega, \Omega) (=\Hom_{R/P} (\Omega, \Omega))$ is the identity element id$_{\Omega}$ in $\Hom_R (\Omega, \Omega)$ (or a unit times id$_{\Omega})$.

{\it Proof of the Claim.}  We can choose $x_1, \ldots, x_{d-1} \in P$, an $R$-sequence, such that $PR_{P}= (x_1, \ldots, x_{d-1}) R_{P}$. Hence the primary decomposition of $(x_1, \ldots, x_{d-1})$ can be written as $P \cap q_2\cap \cdots \cap q_{r}$ where $q_i$ is $P_i$-primary, $\Ht(P_i)={d-1}$, $2 \le i \le r$.  Let $S=R/(x_1, \ldots, x_{d-1})$. Then $(0)=\overline{P}\cap\overline{q_2}\cap \cdots \cap \overline{q_r}$, where $\overline{\nu}$ means the image of $\nu$ in $S$, is a primary decomposition of $(0)$ in $S$.  It can be easily checked that $\Omega_{R/P} =$ the canonical module of $R/P=\ext{R}{d-1}{R/P,R}=\Hom_{S}{}({S/\overline{P},S})=\overline{q}_2\cap \cdots \cap \overline{q}_{r}$.

Let $K_\bullet(\underline{x}; R)$ denote the Koszul complex corresponding to $x_1, ..., x_{d-1}$. Let $\widetilde \alpha : S \to \coker a_{d-1}^*$ be defined by $\widetilde \alpha (\overline 1)$ = im$\tilde e^*$ and let $\alpha_\bullet : K_\bullet(\underline{x}; R) \to L_\bullet ^*$ be a lift of $\widetilde\alpha$. Applying $\Hom (-, R)$ to $\alpha_\bullet$, as above, we obtain a map $\alpha_\bullet^* : L_\bullet \to K^\bullet (\underline x; R)$ which yields a map $\alpha : \Omega \to S$. Since $\Omega$ is a submodule of $R/P$, for any $\mu$ in $\Omega$, $\Omega/ (\mu)$ has codimension $\ge d$. Hence $\alpha$ is injective.
Consider the short exact sequence :
$0 \rightarrow {\Omega} \stackrel{\alpha} \rightarrow S \rightarrow {S/ \Omega}  \rightarrow 0$.

By construction $\Ext_R^{d-1} (S/ \Omega, R) = \Hom_S (S/ \Omega, S) = \overline P$. Applying

$\Hom_R(-, R)$ or $\Hom_S (-, S)$ to the above exact sequence we obtain an injection $\widetilde \beta : R/P = S/ \overline P \hookrightarrow \Hom_S (\Omega, \Omega) = \Ext_R^{d-1} (\Omega, R) \subset \coker a_{d-1}^*$ such that $\widetilde \beta (\overline 1) = $id$_\Omega$=im$\widetilde e^*$. Hence the claim.

Now let us consider $\widetilde \beta : R/P \hookrightarrow \Ext^{d-1} (\Omega, R) \subset \coker a_{d-1}^*$. Let $\beta_\bullet^* : P_\bullet \to L_\bullet^*$, where $\beta_i^* : R^{s_i} \to R^{\gamma_{d-1-i}^*}$ and $\beta_{d-1}^* (1) = \widetilde e^*$, denote a lift of $\widetilde \beta$. Since $\Omega$ is the canonical module of $R/P$, $\coker \widetilde \beta$ has $\grade \ge d+1$. It follows that the corresponding dual map $\beta : \Omega = \Ext^{d-1} (\coker a_{d-1}^* , R) = \Ext^{d-1} (\Ext^{d-1} (\Omega, R), R) \to \Ext^{d-1} (R/P, R) = \Omega \subset \widetilde G_{d-1}$ is an isomorphism. Hence applying $\Hom(-, R)$ to $\beta_\bullet^* : P_\bullet \to L_\bullet^*$, we obtain $\beta_\bullet : L_\bullet \to P_\bullet^*$, a lift of $\beta : \Omega \hookrightarrow \widetilde G_{d-1}$, such that $\beta_{d-1} (\widetilde e) = 1$ and $\beta_{d-1} (\widetilde e_i) = 0, 1 \le i \le \gamma_{d-1}$. Thus we obtain the following commutative diagram.

\begin{equation}\label{CD2}
\begin{gathered}
\xymatrix@=1em{
L_{\bullet}:  \ar[r] & R^{\gamma_{d-1}} \ar[r]^{a_{d-1}} \ar[d]^{\beta_{d-1}} & R^{\gamma_{d-2}} \ar[r] \ar[d] & \cdots \ar[r]  & R^{\gamma_{0}} \ar[r] \ar[d]& \Omega \ar[r]  \ar[d]^{\beta} & 0 \\
P_{\bullet}^*:  0 \ar[r]&  R \ar[r] \ar[d]^{\eta_0^*} & R^{s_{1}^*} \ar[r] \ar[d] & \cdots \ar[r] & R^{s_{d-1}^*} \ar[r] \ar[d]^{\eta_{d-1}^*} & \widetilde{G}_{d-1} \ar[r] \ar[d]^{\eta'} &  0 \\
F_{\bullet}:   \ar[r] & R^{t_d} \ar[r] & R^{t_{d-1}} \ar[r] & \cdots \ar[r] & R^{t_{1}} \ar[r] & S_1 \hookrightarrow  R^{t_0} \ar[r] & M \ar[r] & 0 \\
}
\end{gathered}
\end{equation}
where $\eta': \widetilde{G}_{d-1}\rightarrow S_1$ is induced by $\eta_{d-1}^*$.
\medskip

Since $\eta_{\bullet}^* \circ \beta_{\bullet}$ is a lift of $\eta' \circ \beta =0$, $\eta_{\bullet}^* \circ \beta_{\bullet}$ is homomorphic to 0.  By construction $\eta_{0}^* \circ \beta_{d-1}(\widetilde{e})=e$, which is part of a basis of $R^{t_d}$.  Since both $F_{\bullet}$, $L_{\bullet}$ are minimal free resolutions, this leads to a contradiction. Hence $\Ht(I)\geq d$.

$(1) \Rightarrow (4)$. This has already been mentioned as a special case of a consequence of the order ideal conjecture [2].

$(4) \Rightarrow (3)$.  We first note that by Corollary 1.6 in \cite{D7} the statement in $(4)$ is equivalent to the following: let $\Omega$ denote the canonical module of $R/I$ where $I$ is an almost complete intersection ideal of height $d$. Then $\syz{d}{\Omega}$ has a free summand.
	
	Let $P$ be a prime ideal of height $d$. We can choose $x_1, \ldots, x_d \in P$, an $R$-sequence, such that $PR_{P}= (x_1, \ldots, x_d) R_{P}$. Hence the primary decomposition of $(x_1, \ldots, x_d)$ can be written as $P \cap q_2\cap \cdots \cap q_{r}$ where $q_i$ is $P_i$-primary, $\Ht(P_i)=d$, $2 \le i \le r$.  Let $S=R/(x_1, \ldots, x_d)$. Then $(0)=\overline{P}\cap\overline{q_2}\cap \cdots \cap \overline{q_r}$, where $\overline{\nu}$ means the image of $\nu$ in $S$, is a primary decomposition of $(0)$ in $S$.  It can be easily checked that $\Omega_{R/P} =$ the canonical module of $R/P=\ext{R}{d}{R/P,R}=\Hom_{S}{}({S/\overline{P},S})=\overline{q}_2\cap \cdots \cap \overline{q}_{r}$.

Let us choose $\lambda \in P-\bigcup_{i \ge 2} P_i$. Then $\Omega_{S/\overline{\lambda}S}=\Hom_{S}{}({S/\overline{\lambda}S, S})=\Hom_{S}{}({S/\overline{P},S})=\Omega_{R/P}$. \ Thus, \ the almost complete intersection ideal $(x_1, \ldots, x_d, \lambda)\subset P$ is such that $\Omega_{R/{(x_1, \ldots, x_d, \lambda)}}=\Omega_{R/P}$. Hence $\Syz^d (\Omega_{R/P})$ has a free summand.
\end{proof}

\begin{corollary}
 Let $(R, m, k)$ be a regular local ring. For the validity of the order ideal conjecture it is enough check the validity of the assertion for minimal generators of $\syz{d+1}{M}$ for any finitely generated $R$-module $M$ of codimension $d \ge 0$, in particular for cyclic modules $R/I$ where $I$ is an almost complete intersection ideal of height d.
\end{corollary}

\subsection*{}

\textbf{2.2} Let $(R,m,k)$ be a local ring of dimension $n$ and let $I$ be an ideal of $R$ of grade $d$.  Due to the associativity property of $\Hom$ and tensor product we obtain two spectral sequences converging to the same limit whose $E_2^{i,j}$ terms are: $\ext{R}{i}{\tor{j}{R}{k, R/I},R}$ and $\ext{R}{i}{k,\ext{}{j}{R/I,R}}$.


Since grade of $I=d$, $\ext{}{j}{R/I,R}=0$ for $j<d$.  Thus, the standard edge homomorphism reduces to: $\ext{R}{i}{k,\ext{}{d}{R/I,R}} \rightarrow (i+d)$-th limit.  When $R$ is a regular local ring of dimension $n$ and $i+j=n$, this limit is $\ext{R}{n}{k,R}$ and the first $E_2^{i,j}$ mentioned above collapses to this term.  In this situation the following question is raised: when is the edge homomorphism $\eta_d: \ext{R}{n-d}{k, \ext{}{d}{R/I,R}}\rightarrow \ext{R}{n}{k,R}$ non-zero? Our next theorem answers this question in relation to the two most important homological conjectures in commutative algebra.

\begin{theorem*}
The following statements are equivalent:
\begin{enumerate}
\item[(1)]	For every regular local ring $R$ of dimension $n$ and for any ideal $I$ of height $d$ of $R$, the edge homomorphism $\eta_d: \ext{R}{n-d}{k, \ext{}{d}{R/I,R}}\rightarrow \ext{R}{n}{k,R}$ is non-zero.
\item[(2)] The same statement where $I$ is replaced by a prime ideal $P$ of $R$.
\item[(3)] The order ideal conjecture is valid over regular local rings.
\item[(4)] The monomial conjecture is valid for all local rings.
\end{enumerate}
\end{theorem*}

\begin{proof}
We will prove that $(1) \Rightarrow (2) \Rightarrow (3) \Rightarrow (4) \Rightarrow (1)$.
$(1) \Rightarrow (2)$ is obvious.

$(2) \Rightarrow (3)$.  By Theorem 2.1, in order to prove the order ideal conjecture on a regular local ring $R$, we need to show that for every prime ideal $P$ of height $d$, if $\Omega=\ext{R}{d}{R/P, R}$ (the canonical module of $R/P$), then $\syz{d}{\Omega}$ has a free summand.  Let $x_1$, $\ldots$, $x_d$ be an $R$-sequence contained in $P$.  Let $S=R/(x_1,\ldots, x_d)$.  Then $\Omega=\Hom_{R}{}({R/P, S})=\Hom_{S}{}{(R/P,S})$; let $i: \Omega \hookrightarrow S$ denote the natural inclusion.  Let $K_{\bullet}=K_{\bullet}(x_1, \ldots, x_d; R)$, $L_{\bullet}$ denote the Koszul complex with respect to $x_1$, $\ldots$, $x_d$ and a minimal free resolution of $\Omega$ over $R$ respectively and let $f_{\bullet}: L_{\bullet} \rightarrow K_{\bullet}$ be a lift of $i: \Omega \hookrightarrow S$.  It has been pointed out in Sections 1.5 and 1.6 of \cite{D7} that $\syz{d}{\Omega}$ has a free summand if and only if $f_d(L_d)=R=K_d(\ul{x}; R)$, and hence if and only if $\tor{d}{R}{k,\Omega}\rightarrow\tor{d}{R}{k,S}$ is non-zero.

Since $R$ is a regular local ring of dimension $n$, the maximal ideal $m$ is generated by a regular system of parameters $y_1$, $\ldots$, $y_n$.  Hence $K_{\bullet}(\ul{y}; R)$, the Koszul complex with respect to $y_1$, $\ldots$, $y_n$, is a minimal free resolution of $k=R/m$.  Thus, ${\tilde i}: \tor{d}{R}{k,\Omega} \rightarrow \tor{d}{R}{k,S} \ne 0$ if and only if ${\tilde i}: \ext{R}{n-d}{k, \Omega} \rightarrow \ext{R}{n-d}{k,S}$ is non-zero.  For any local ring $(A, m, k)$, let $E_A$ denote the injective hull of the residue field $k$.  Let $F_{\bullet} (F_i = R^{t_i})$ be a minimal free resolution of $R/P$ and let $I^{\bullet}$ be a minimal injective resolution of $R$. Let us consider the double complex $D^{\bullet \bullet} = \Hom_R(F_\bullet, I^\bullet)$. Since height of $P$ = d, $H^i(D^{\bullet \bullet}) = 0$ for $i<d$, $H^d(D^{\bullet \bullet}) = \Omega$ and $H^i(D^{\bullet \bullet}) = \Ext^i(R/P, R)$ for $i>d$. After deleting the initial exact part of $D^{\bullet \bullet}$ we obtain a minimal complex of injective modules $(T^\bullet, \psi^\bullet) (\{ T^i\}_{i \ge 0})$ over R such that $T^i = \bigoplus_{r+s = d+i} D^{r, s}$ for $i > 0$, $H^0(T^\bullet) = \Omega$ and $H^i (T^\bullet) = \Ext ^{d+i} (R/P, R)$, for $i > 0$. By proposition 1.1 in [5] we can construct a minimal injective complex $J^\bullet ( \{J^i\} _{i \ge 0})$ such that $H^i(J^\bullet) \simeq \Ext^{d+1+i}(R/P, R)$ for $i \ge 0$, and a map $\phi^\bullet$ : $T^\bullet \rightarrow J^\bullet (\phi^i : T^{i+1} \rightarrow J^i, i \ge 0)$ such that $\phi^\bullet$ induces an isomorphism between $H^{i+1}(T^\bullet)$ and $H^i(J^\bullet)$ for $i \ge 0$. The mapping cone of $\phi^\bullet$ provides an injective resolution of $\Omega$ over R. We have the following commutative diagram


\begin{equation}\label{CD1}
\begin{gathered}
\xymatrix@=1em{
 0 \ar[r]  & T^{0} \ar[r] & T^{1} \ar[r] \ar[d]^{\phi^0} &T^{2} \ar[r] \ar[d]^{\phi^{1}} & \cdots \ar[r] & T^{n-d} \ar[r]^{\psi^{n-d}} \ar[d]^{\phi^{n-d-1}} & T^{n-d+1} \ar[r] \ar[d]^{\phi^{n-d}} & \cdots \\
   & 0 \ar[r] & J^{0} \ar[r] & J^{1} \ar[r] & \cdots \ar[r] & J^{n-d-1} \ar[r] & J^{n-d} \ar[r] & \cdots
}
\end{gathered}
\end{equation}

Applying $H_m^0(-)$ (the 0-th local cohomology functor with respect to the maximal ideal $m$) to this diagram we obtain the following commutative diagram

\begin{equation}\label{CD1}
\begin{gathered}
\xymatrix@=1em{
 0 \ar[r] & 0 \ar[r] & 0 \ar[r] \ar[d] & 0 \ar[r] \ar[d] & \cdots 0 \ar[r] \ar[d] & E \ar[r]^{{\tilde \psi}^{n-d}} \ar[d]^{{\tilde \phi}^{n-d-1}} & E^{t_1} \ar[r] \ar[d]^{{\tilde \phi}^{n-d}} & \cdots \\
   & 0 \ar[r] & E^{r_0} \ar[r] & E^{r_1} \ar[r] & \cdots \ar[r] & E^{r_{n-d-1}} \ar[r] & E^{r_{n-d}} \ar[r] & \cdots
}
\end{gathered}
\end{equation}

where $E$ denotes the injective hull of the residue field k over R, ${\tilde \psi}$ etc. denote the corresponding restriction maps and $H_m^0 (J^i) = E^{r_i}$, for $i \ge 0$.

  Now it can be checked that $\eta_d: \ext{R}{n-d}{k,\Omega}\rightarrow \ext{R}{n}{k, R} \ne 0 \Leftrightarrow {\tilde \phi}^{n-d-1}$ in diagram (4) is not injective $\Leftrightarrow$ the composite map $\ext{R}{n-d}{k, \Omega}\rightarrow \hm{m}{n-d}{\Omega} \rightarrow E_{R} \ne 0 \Leftrightarrow$ the composite map $\ext{R}{n-d}{k,\Omega}\rightarrow\hm{m}{n-d}{\Omega} \rightarrow \Hom_R (R/P, E_R) \hookrightarrow E_S (\hookrightarrow E_{R})$ is non-zero. The last equivalence follows from the observations that ker$({\tilde \psi}^{n-d}) = \Hom_R (R/P, E_R)$ and $H_m^{n-d}(\Omega) \cong \Hom_R(\Ext^d(\Omega, R), E_R)$ (Grothendieck duality) $\to \Hom_R (R/P, E_R)$ is an onto map.
Since $x_1$, $\ldots$, $x_d$ form an $R$-sequence and $R$ is regular it follows that $\ext{R}{n}{k,R} \cong \ext{R}{n-d}{k,S}$.  We have the following commutative diagram via $i: \Omega \hookrightarrow S$.
\[
\xymatrix@=1em{
\ext{R}{n-d}{k,\Omega} \ar[r] \ar[dr] & \ext{R}{n-d}{k,S} = k \hookrightarrow E_S \ar@{=}[r] &  \hm{m}{n-d}{S} \\
 & \hm{m}{n-d}{\Omega} \ar[ur] &
}
\]
Thus $\eta_d\ne 0$ if and only if the composite of the bottom maps is non-zero, i.e., if and only if $\widetilde{i}: \ext{R}{n-d}{k,\Omega} \rightarrow \ext{R}{n-d}{k,S}$ is non-zero.

That the spectral sequence map $\eta_d: \ext{R}{n-d}{k,\Omega} \rightarrow \ext{R}{n}{k,R} (=k)$ boils down to $\widetilde{i}: \ext{R}{n-d}{k,\Omega} \rightarrow \ext{R}{n-d}{k,S} (=k)$ can be checked by comparing the double complexes $\Hom_{R}{}({F_{\bullet},J^{\bullet}})$ and $\Hom_{R}{}({K_{\bullet}(\ul{x};R),J^{\bullet})}$ via the natural surjection $S=R/(\ul{x})\rightarrow R/P$ after applying $H_m^0(-)$ to each one of them.

$(3) \Rightarrow (4)$.   This follows readily from Theorem 1.4 in \cite{D7}.

$(4) \Rightarrow (1)$.  Since the monomial conjecture is valid for all local rings, the canonical element conjecture is also valid for all local rings. Thus, if $(A,m,k)$ is a local ring of dimension $t$ such that $A$ is the image of a Gorenstein ring and $\Omega$ its canonical module, the composite of $\ext{A}{t}{k,\Omega} \rightarrow \hm{m}{t}{\Omega}\rightarrow E_A$ is non-zero \cite{H2}.

Let $(R,m,k)$ be a regular local ring and let $I$ be an ideal of height $d>0$.  Let $A=R/I$, $\Omega=\ext{R}{d}{R/I,R}$, the canonical module of $A$.  We are required to prove that the edge homomorphism map $\eta_d: \ext{R}{n-d}{k,\Omega}\rightarrow \ext{R}{n}{k,R}$ is non-null.  Let $F_{\bullet}$, $J^{\bullet}$ denote the minimal projective resolution of $R/I$ and the minimal injective resolution of $R$ over itself respectively.  Considering the double-complex $\Hom_{R}{}({F_{\bullet}, J^{\bullet}})$ and arguing as in the previous part, we observe that $\eta_d\ne 0$ if and only if the composite map $\ext{R}{n-d}{k,\Omega}\rightarrow\hm{m}{n-d}{\Omega}\rightarrow E_{R}$ is non-null.  Since $A$ satisfies the canonical element conjecture the composite map $\ext{A}{n-d}{k,\Omega}\rightarrow\hm{m}{n-d}{A}\rightarrow E_A (=\Hom_{R}{}({A,E_{R}})) \ne 0$.

Hence, from the commutative diagram
\[
\begin{CD}
\ext{A}{n-d}{k, \Omega} @>>> \hm{m}{n-d}{\Omega} @>>> E_{A}\\
@VVV  @| @VVV \\
\ext{R}{n-d}{k,\Omega} @>>>\hm{m}{n-d}{\Omega} @>>> E_{R} \\
\end{CD}
\]
It follows that $\eta_d\ne0$ and our proof is complete.
\end{proof}

{\textbf 2.3} As a consequence of the above theorem we obtain the following.

\begin{theorem*}
Let $(R,m,k)$ be an equicharacteristic regular local ring of dimension n. Let I be an ideal of R
of height d. Then the edge homomorphism $\eta_d : \Ext_R^{n-d} (k, \Ext^d (R/I, R)) \to \Ext_R^n (k, R)$ is non-zero.
\end{theorem*}

The proof follows readily from the validity of the monomial conjecture/order ideal conjecture over equicharacteristic local rings.

\begin{remark} The equivalence of the monomial conjecture and the order ideal conjecture on regular local rings implies the stronger nature of the validity of the order ideal conjecture for all local rings.  We refer the reader to \cite{D8} and \cite{D9} for relevant observations/reduction and proofs of special cases of this general statement of the order ideal conjecture in mixed characteristic.
\end{remark}

\section*{Section 3} \label{s3}

\subsection*{}

\textbf 3.1 As remarked earlier that for the validity of the order ideal conjecture over a regular local ring $R$, it is enough to check the assertion of the conjecture for minimal generators of $\syz{d+1}{M}$ for finitely generated $R$-modules $M$ of codimension $d$.  At present we are able to prove the following.

\begin{theorem*}
Let $(R,m,k)$ be a regular local ring of dimension $n$ and let $M$ be a finitely general $R$-module of codimension $d$.  Suppose that $M$ satisfies Serre-condition $S_1$and is equidimensional.  Then, for every minimal generator $\beta$ of $S_i(M)(=\syz{i}{M})$, $\gr(\ring{S_i(M)}{\beta}) \ge i$, for $i \ge d$.
\end{theorem*}

\begin{proof}
 We consider a minimal free resolution of $M$:
\[
(F_{\bullet},\phi_{\bullet}):\rightarrow R^{t_{i+1}}\stackrel{\phi_{i+1}}{\rightarrow} R^{t_{i}} \stackrel{\phi_{i}}{\rightarrow} R^{t_{i-1}} \rightarrow \cdots \rightarrow R^{t_{1}} \stackrel{\phi_{0}}{\rightarrow} R^{t_{0}} \rightarrow
M \rightarrow 0.
\]
Assume $i \ge d$.  Let $\beta$ be a minimal generator of $S_i(M)$.  Let $\{e, e_j\}_{1 \le j \le t_i -1}$ be a basis of $R^{t_i}$ such that $\phi_i(e)=\beta$.  Let $\beta=
(a_1, \ldots, a_{t_{i-1}}) \in R^{t_{i-1}}$ and let $I$ denote the ideal generated by $a_1$, $\ldots$, $a_{t_{i-1}}$.  We assert that the grade of $I$ (which equals the height of $I$) is greater than or equal to $i$.  If possible, let $\Ht(I) \le i-1$.  By induction on $i$, since $S_i(M)=S_{i-1}(S_1(M))$, we can assume $\Ht(I)=i-1$.  Let $P$ be a minimal prime ideal of $I$ such that $\Ht(P)=\Ht(I)=i-1$. Let $G_i=\coker(\phi_i^*)$.  Let $F_{\bullet}^*$ denote the complex
\[
0 \rightarrow R^{t_0^*} \rightarrow R^{t_1^*} \rightarrow \cdots \rightarrow R^{t_{i-1}^*} \rightarrow R^{t_i^*} \rightarrow G_i \rightarrow 0,
\]
and let $\pi': R^{t_i^*} \rightarrow R$ be the projection onto the $e^*$-th component. Then $\pi' \circ \phi_i^*=I$.  Let $(Q_{\bullet}, \psi_{\bullet})$ be a minimal free resolution of $R/P$.  We have the following commutative diagram,
\begin{equation}\label{CD3}
\begin{gathered}
\xymatrix@=1em{
F_{\bullet}^*: 0 \ar[r] & R^{t_0^*} \ar[r]^{\phi_0^*} \ar[d]^{\pi_0} & R^{t_1^*} \ar[r] \ar[d]^{\pi_1} & \cdots \ar[r] & R^{t_{i-1}^*} \ar[r]^{\phi_{i}^*} \ar[d]^{\pi_{i-1}} & R^{t_i^*}\ar[r] \ar[d]^{\pi'} & G_i \ar[r] \ar[d]^{\pi} & 0 \\
Q_{\bullet}:  \ar[r] & R^{s_i} \ar[r]^{\psi_i} & R^{s_{i-1}} \ar[r]^{\psi_{d-1}} & \cdots \ar[r] & R^{s_1} \ar[r]^{\psi_1}& R \ar[r] & R/P \ar[r]  & 0
}
\end{gathered}
\end{equation}
where $\pi_{\bullet}$ is a lift of $\pi: G_i \rightarrow R/P$ induced by $\pi'$.  Applying $\Hom_R (-,R)$ to \eqref{CD3} we obtain the following commutaive diagram,
\begin{equation}\label{CD4}
\begin{gathered}
\xymatrix@=1em{
Q_{\bullet,i}^*: 0 \ar[r] & R \ar[r] \ar[d]^{\pi{'*}} & R^{s_1^*} \ar[r] \ar[d]^{\pi_{i-1}^*} & \cdots \ar[r] & R^{s_{i-1}^*} \ar[r]^{\psi_{i}^*} \ar[d]^{\pi_1^*} & R^{s_i^*} \ar[r] \ar[d]^{\pi_0^*} & \widetilde{G}_i \ar[r]\ar[d]^{\widetilde{\pi}} & 0 \\
F:  \ar[r] & R^{t_i} \ar[r] & R^{t_{i-1}}\ar[r] & \cdots \ar[r] & R^{t_1} \ar[r] & R^{t_0} \ar[r] & M \ar[r] & 0 \\
}
\end{gathered}
\end{equation}
where $\widetilde{G}_i=\coker \psi_i^*$ and $\widetilde{\pi}$ is induced by $\pi^*$.

Let $\Omega=\ext{R}{i-1}{R/P,R}$.  Since $\Ht(P)=i-1$, $\ext{}{j}{R/P,R}=0$ for $j < i-1$ and $\Omega \ne 0$.  Let $\alpha: \Omega \hookrightarrow \coker \psi_{i-1}^*$ denote the natural inclusion.  $\pi_1^*$ maps $\coker \psi_{i-1}^*$ to $S_1(M)$; since $P$ annihilates $\Omega$, $\pi_1^* \circ \alpha=0$.  Since $\Ht(P)=i-1$, the codimension of $\ext{}{i}{R/P,R}>i$.  Let $\gamma:\ext{}{i}{R/P,R}\hookrightarrow\widetilde{G}_i$  denote the natural inclusion.  Since $M$ satisfies Serre-condition $S_1$ and is equidimensional of codimension $d$ and $i \ge d$, we have $\widetilde{\pi}(\ext{}{i}{R/P,R})=0$.

Now applying Proposition 1.1 in \cite{D1} we can construct a minimal free complex $(L_{\bullet}, \theta_{\bullet})$ such that $H_0(L_{\bullet})=\ext{}{i}{R/P,R}$, $H_1(L_{\bullet})=\ext{}{i-1}{R/P,R}=\Omega$ and a map $g_{\bullet}: L_{\bullet} \rightarrow Q_{\bullet,i}^*$ such that it induces the inclusion map $\gamma: \ext{}{i}{R/P,R} \hookrightarrow \coker \psi_i^*$ and isomorphism on higher homologies.  Then the mapping cone $V_{\bullet}$ of $g_{\bullet}$ is a free resolution of $\im \psi_{i+1}^*$.  We have the following commutative diagram.
\begin{equation}\label{CD5}
\begin{gathered}
\xymatrix@=1em{
L_{\bullet}:  \ar[r] & R^{p_i} \ar[r] \ar[d]^{g_i} & R^{p_{i-1}} \ar[r] \ar[d]^{g_{i-1}}& \cdots \ar[r] & R^{p_{1}} \ar[r] \ar[d]^{g_1}& R^{p_0} \ar[r] \ar[d]^{g_0}& \hm{0}{}{L_{\bullet}}\ar[r] \ar[d]^{\gamma} & 0 \\
Q_{\bullet,i}^*:  0 \ar[r] &  R  \ar[r] &  R^{s_1^*}\ar[r] &  \cdots \ar[r] &  R^{s_{i-1}} \ar[r] &  R^{s_i} \ar[r] &  \widetilde{G}_i \ar[r] &  0  \\
}
\end{gathered}
\end{equation}

\begin{claim}
$\im g_i=R$.
\end{claim}

If not, then the copy of $R$ in $Q_{\bullet,i}^*$ survives in a minimal free resolution $U_{\bullet}$ of $\im \psi_{i+1}^*$ extracted from $V_\bullet$.  Since $\Ht(P)=i-1$, from \eqref{CD5} it follows that $\syz{i}{\im \psi_{i+1}^*}$ has a minimal generator whose entries generate an ideal of height less than $i$. Since $\im \psi_{i+1}^* \subset R^{s_{i+1}}$, the mixed characteristic $p(>0)$  is a non-zero divisor on $\im \psi_{i+1}^*$.  Since the order ideal conjecture is valid over equicharacteristic local rings, we arrive at a contradiction by tensoring $U_{\bullet}$  with $R/pR$.  Hence $\im g_i = R$.

Let $e'$ be a part of a basis of $L_i ( = R^{p_i})$ such that $g_i(e')=1$.  By construction $\pi_{\bullet}^* \circ g_{\bullet}$ is a lift of ${\widetilde \pi} \circ \gamma$. Since ${\widetilde \pi} (\Ext^i (R/P, R))=0$, we have ${\widetilde \pi} \circ \gamma = 0$. Hence $\pi_{\bullet}^* \circ g_{\bullet}$ is homotopic to $0$.  Since $L_{\bullet}$ and $F_{\bullet}$ are minimal free complexes and $\pi'^* \circ g_i(e')=e$, a part of a basis of $R^{t_i}=F_i$, we arrive at a contradiction. Hence $\Ht(I)\ge i$, for $i \geq d$.
\end{proof}

\subsection*{}
\textbf{3.2} In our next proposition we study a special case of order ideals of syzygies of modules, not necessarily of finite projective dimension, over any Cohen-Macaulay local ring.

\textbf {Proposition.} \textit {Let $(R, m, k)$ be a Cohen-Macaulay local ring of dimension $n$ and let $M$ be a finitely generated $R$-module.  Let $S_i$ denote the $i$-th syzygy of $M$ for $j>0$. Let $\alpha$ be a minimal generator of $S_i$ such that it generates a free summand of $S_i$.  Then $\Ht(\ring{S_i}{\alpha})\geq i$, for $i \le n-1$,  and $\Ht(\ring{S_i}{\alpha})=n$ for $i \geq n$}.

\begin{proof} Let $(F_\bullet, \phi_\bullet) \{ R^{r_i}, \phi_{i} : R^{r_i} \to R^{r_{i-1}} \}$ be a minimal free resolution of $M$. Let $i = 1$. Since $S_1$ has a free summand generated by $\alpha$ we can write $S_1 = R \alpha \oplus T_1$. Let $\alpha_1 : R \hookrightarrow R^{r_0}$ denote the restriction of the inclusion map : $S_1 \hookrightarrow R^{r_0}$. Let $\alpha_1 (1) = (a_1, \dots, a_{r_0})$. Since coker$\alpha_1$ has projective dimension 1, we have grade $\Ext^1$(coker$\alpha_1, R)$ is $\ge 1$. Thus grade of the ideal generated by $a_1, \dots, a_{r_0}$ is greater than or equal to 1. Now we use induction on $i$. Let $S_i = R \alpha \oplus T_i$. Let $\alpha_i : R \hookrightarrow R^{r_{i-1}}$ denote the restriction of the inclusion map $S_i \hookrightarrow R^{r_{i-1}}$. Let $I$ denote the ideal generated by the entries of $\alpha_i(1)$. Then, by above, grade $I \ge 1$. Let $x \in I$ be a non-zero-divisor on $R$ and let $( \bar \ )$ denote the image modulo $xR$. Then $\overline S_i = \Syz^{i-1}(\overline S_1)$. By induction, grade$(\overline I) \ge {i-1}$. Hence grade $I \ge i$.
\end{proof}

\subsection*{}

\textbf {3.3} The following theorem points out a connection between $d^{th}$ syzygies of $R/P$ and its canonical module in terms of having a free summand.

\begin{theorem*}\label{Tm3.3}
Let $(R,m,k)$ be a complete regular local ring of dimension $n$ and let $P$ be a prime ideal of height $d$, $\Omega=\ext{}{d}{R/P,R}$ –--- the canonical module for $R/P$.  Suppose $\syz{d}{R/P}$ has a free summand.  Let $x_1,\ldots,x_d \in P$ form an $R$-sequence; let $S=R/(x_1,\ldots,x_d )$.  Then there exists $\mu \in S$ and a map $\alpha: R/P \hookrightarrow S$,  $\alpha(\bar 1)=\mu$ such that the map : $\ext{R}{n-d}{k, R/P}\rightarrow \ext{R}{n-d}{k,S}$ induced by $\alpha$ is non-null.  Moreover, $\syz{d}{\Omega}$ has a free summand.
\end{theorem*}

\begin{proof}
Let $(F_{\bullet},\phi_{\bullet})$ be a minimal free resolution of $R/P$.
\[
F_\bullet: \rightarrow R^{t_{d+1}} \stackrel{\phi_{d+1}}{\rightarrow}R^{t_{d}} \stackrel{\phi_{d}}{\rightarrow} R^{t_{d-1}} \rightarrow\cdots \rightarrow R^{t_{1}} \stackrel{\phi_{1}}{\rightarrow} R \rightarrow R/P \rightarrow 0.
\]
Let $F_{\bullet}^* = \Hom_{R}{}({F_{\bullet}, R})$ and
\[
F_{\bullet,d}^*:0\rightarrow R \rightarrow R^{t_1^*}\rightarrow \cdots\rightarrow R^{t_{d-1}^*}\rightarrow R^{t_d^*}\rightarrow G_d \rightarrow 0,
\]
where $G_d=\coker (\phi_d^*)$.  Since $\syz{d}{R/P}$ has a free summand, $R^{t_d}$ has a basis $\{e, e_i\}_{1 \le i \le t_d-1}$ such that im $\phi_{d+1}$ is contained in the submodule generated by $\{ e_i\}_{1 \le i \le t_d-1}$ and $\phi_{d+1}^*(e^*)=0$.  Then im $e^* \in \Omega \hookrightarrow G_d$.  By construction, $\Omega=\ext{}{d}{R/P,R}=\Hom_{S}{}({R/P,S})$.  Let $K_{\bullet}(\ul{x};R)$ denote the Koszul complex corresponding to $x_1$, $\ldots$, $x_d$.  Consider the map $\eta: S \rightarrow\Omega \subset G_d$, $\eta(1)=$ im $e^*$.  Let $\eta_{\bullet}: K_{\bullet}(\ul{x};R)\rightarrow F_{\bullet,d}^*$ denote a lift of $\eta$.  We have the following commutative diagram.
\begin{equation}\label{CD6}
\begin{gathered}
\xymatrix@=1em{
K_{\bullet}: 0 \ar[r] & R \ar[r]  \ar[d]^{\eta_d= {\mu}}  & R^{d} \ar[r] \ar[d]^{\eta_{d-1}}& \cdots \ar[r]   & R^{d}\ar[r] \ar[d]^{\eta_{1}}& R \ar[r] \ar[d]^{\eta_{0}} & R/(\ul{x}) = S \ar[d]^{\eta}  \\
F_{\bullet, d}^*:  0 \ar[r] & R  \ar[r] & R^{t_1^*} \ar[r] & \cdots \ar[r] & R^{t_{d-1}^*} \ar[r] & R^{t_d^*} \ar[r] & G_d \ar[r] & 0  \\
}
\end{gathered}
\end{equation}
Since $P$ is a prime ideal of height $d$, $\ext{}{i}{R/P,R}=0$ for $i<d$ and $\Ext_R^i(\Ext_R^i$ $(R/P,R),R)=0$ for $i>d$.  Applying $\Hom_{R}{}({-,R})$ to \eqref{CD6} we obtain the following: $\eta_d^*$ induces a map $\alpha: R/P=\ext{}{d}{G_d,R}\hookrightarrow S$, $\alpha(\overline{1})=\mu$ and $\eta_0^*(e)=1\in R=K_d(\ul{x}; R)$.  Since $\mu P=0$ in $S$, $\mu \in \Omega$ and hence $\alpha$ factors through $R/P \stackrel{\beta}{\hookrightarrow} \Omega \stackrel{\gamma}{\hookrightarrow} S$ where $\beta(\overline{1})=\mu$ and $\gamma$ is the natural inclusion. Let $L_{\bullet}$ be a minimal free resolution of $\Omega$. Let $\beta_{\bullet}:F_{\bullet}\rightarrow L_{\bullet}$ and $\gamma_{\bullet}:L_{\bullet}\rightarrow K_{\bullet}(\ul{x},R)$ denote lifts of $\beta$ and $\gamma$.  Then $\gamma_{\bullet}\circ\beta_{\bullet}$ is homotopic to $\eta^*$.  Since $\eta_0^*(R^{t_d})=R$, we have $\gamma_d(L_d)=R$.  Thus $\syz{d}{\Omega}$ has a free summand.

Since $\ext{R}{n-d}{k,R/P}=\tor{d}{R}{k,R/P}$, $\ext{R}{n-d}{k,S}=\tor{d}{R}{k,S}$ ($R$ being regular local), $\eta_0^*(e)=1$, it follows that $\ext{R}{n-d}{k,R/P}\rightarrow \ext{R}{n-d}{k,S}$, induced by $\alpha$, is non-null.
\end{proof}

\begin{corollary}
Let $R$, $P$ be as above in the theorem and let us assume that $\syz{d}{R/P}$ has a free summand.  Then the direct limit map $\ext{R}{n-d}{k,R/P}\rightarrow \hm{m}{n-d}{R/P}$ is non-zero.
\end{corollary}

This follows from our theorem and the following commutative diagram.
\[
\begin{CD}
\ext{R}{n-d}{k, R/P} @>>> \ext{R}{n-d}{k,S} = k \\
@VVV  \hookdownarrow  \\
\hm{m}{n-d}{R/P} @>>> \hm{m}{n-d}{S}\\
\end{CD}
\]
\begin{remark}  Let $A$ be a complete local normal domain of dimension $t$.  In Theorem 3.8 of \cite{D4} we showed that if $\Omega\cap(m_S-m_S^2)\ne \emptyset$, then the direct limit map $i: \ext{A}{t}{k,A} \rightarrow \hm{m}{t}{A}$ is non-null.  However we couldn't figure out the real significance of non-null property of the direct limit map $i$.  The above theorem shows that if $i$ is non-null and $A=R/P$ where $R$ is a complete regular local ring, $d = \Ht(P)$, then $\syz{d}{R/P}$ must have a free summand.  For details we refer the reader to \cite[Section 3]{D4}.
\end{remark}

\subsection*{}

\textbf {3.4} Our final theorem proves another special case of the order ideal/monomial conjecture.

\begin{theorem*}\label{Tm3.4} Let $(R,m,k)$ be a regular local ring of dimension $n$ and let $P$ be a prime ideal of height $d$.  We assume $R/P$ satisfies Serre-condition $S_4$ and $\nu(\Omega) \le 2$, where $\nu(\Omega)$ is the minimal number of generators of the canonical module $\Omega$ of $R/P$.  Then $\syz{d}{\Omega}$ has a free summand.
\end{theorem*}

\begin{proof}
If $\nu(\Omega)=1$ and $R/P$ satisfies Serre-condition $S_3$, the assertion follows from Theorem 2.6 \cite{D4} and Theorem \ref{Tm2.2}.  Let $\{\mu_1, \mu_2\}$ be a minimal set of generators of $\Omega$.  As pointed out in the proof of $(4)\Rightarrow(3)$ in Theorem \ref{Tm2.1}, we have $\Omega \hookrightarrow R/P$.  Then, for $\mu = \mu_1$ or $\mu_2$, we have $\Omega/(\mu)\hookrightarrow R/(P+\mu R)$.  Since $R/P$ satisfies Serre-condition $S_2$, $R/(P+\mu R)$ satisfies $S_1$ and hence $\Omega/(\mu)\cong R/I$, where $\Ht I = d+1$.
We consider a minimal free resolution $(F_{\bullet},\phi_{\bullet})$ of $\Omega$
\[
F_{\bullet}: \rightarrow R^{t_{d+2}} \stackrel{\phi_{d+2}}{\rightarrow}R^{t_{d+1}} \stackrel{\phi_{d+1}}{\rightarrow}R^{t_{d}} \stackrel{\phi_{d}}{\rightarrow} \cdots \rightarrow R^{t_{1}} \stackrel{\phi_{1}}{\rightarrow} R^2 \rightarrow \Omega \rightarrow 0.
\]
Since $R/P$ satisfies Serre-condition $S_4$, we have $R/P \cong \Hom_{R}{}({\Omega,\Omega})=\ext{R}{d}{\Omega,R}$, $\ext{}{d+1}{\Omega,R}=0$, $\ext{}{d+2}{\Omega, R}=0$.  We consider the following truncated part of $F_{\bullet}^*=\Hom_R ({F_{\bullet}, R})$ :
\[
F_{\bullet,d+2}^*:  0 \rightarrow R^{2} \stackrel{\phi_{1}^*} {\rightarrow}R^{t_{1}^*} \rightarrow \cdots \rightarrow R^{t_{d}^*} \rightarrow R^{t_{1+1}^*}\rightarrow R^{t_{d+2}^*} \rightarrow \im \phi_{d+3}^* \rightarrow 0.
\]
Let $G_i=\coker \phi_i^*$; then $R/P=\ext{}{d}{\Omega,R} \stackrel{\alpha}{\hookrightarrow} G_d$ and $\coker (\phi_{d+2}^*)=\im \phi_{d+3}^*$.  Let $L_{\bullet}$ be a minimal free resolution of $R/P$ and $f_{\bullet}: L_{\bullet}\rightarrow F_{\bullet, d+2}^*$ be a lift of $\alpha: R/P \hookrightarrow G_d$.

\begin{claim}
$f_d (L_d) = R^2$.
\end{claim}

If possible, let $f_d(L_d) \ne R^2$.  Then there exists a minimal generator $e$ of $R^2$ such that $e \notin f_d(L_d)$.  The mapping cone $V_{\bullet}$ of $f_{\bullet}$ provides a free resolution of $\im \phi_{d+3}^*$.  In the minimal free resolution of $U_{\bullet}$ of $\im \phi_{d+3}^*$ extracted from $V_{\bullet}$, the copy of $e$ survives as a minimal generator of $U_{d+2}$. Since height$(I) = d+1$, the entries of $\phi_1^*(e)$ has height $d+1<d+2$.  Since $\im \phi_{d+3}^*\subset R^{t_{d+3}^*}$, the mixed characteristic $p(>0)$ is a non-zero-divisor on $\im \phi_{d+3}^*$.  Hence tensoring $U_{\bullet}$ with $R/pR$ we arrive at a contradiction due to the validity of the order ideal conjecture in equicharacteristic.

Hence $f_d(L_d)=R^2$.

This implies that $\syz{d}{R/P}$ has a free summand. Hence, by the previous theorem, $\Syz^d (\Omega)$ has a free summand.
\end{proof}

\begin{corollary}
Let $A$ be a complete local domain satisfying Serre-condition $S_4$.  Assume that the canonical module is generated by two elements.  Then $A$ satisfies the monomial conjecture.
\end{corollary}

\begin{remark}
The consequence of the order ideal conjecture due to Bruns and Herzog mentioned in the introduction can be improved.  It has been proved in \cite{D9}, that if $R$ is a regular local ring and $I$ is an ideal of $R$ of height $d$ generated by $\{x_i\}$ such that $R/I$ is $S_2$, then $K_d(\ul{x}; R)\otimes k \rightarrow \tor{d}{R}{R/I,k}$ is the $0$-map.
\end{remark}



\begin{thebibliography}{99}

\bibitem{Bh} B. Bhatt, Almost direct summands, Nagoya Math. J. 214 (2014), 195-204.

\bibitem{BrH} W. Bruns, J. Herzog, \emph{Cohen-Macaulay Rings}, Cambridge Stud. Adv. Math, \textbf{39}, Cambridge Univ. Press, Cambridge, 1993.

\bibitem{D1} S. P. Dutta, On the canonical element conjecture, \emph{Trans. Amer. Math. Soc.} \textbf{299} (1987) 803--811.

\bibitem{D2} S. P. Dutta, Syzygies and homological conjectures, in:  Commutative Algebra, \emph{Math. Sci. Res. Inst. Publ.} \textbf{15} (1989) 139--156.

\bibitem{D3} S. P. Dutta, Dualizing complex and the canonical element conjecture, \emph{J. Lond. Math. Soc.} \textbf{50} no. 2, (1994) 477--487.

\bibitem{D4} S. P. Dutta, Dualizing complex and the canonical element conjecture II, \emph{J. Lond. Math. Soc.} \textbf{2} no. 56, (1997) 46--63.

\bibitem{D5} S. P. Dutta, A note on the monomial conjecture, \emph{Trans. Amer. Math. Soc.} \textbf{350} (1998) 2871--2878.

\bibitem{D6} S. P. Dutta, Splitting of local cohomology of syzygies of the residue field, \emph{J. Algebra} \textbf{244} (2001) 168--185.

\bibitem{D7} S. P. Dutta, Monomial Conjecture and order ideals, \emph{J. Algebra} \textbf{383} (2013) 232--241.

\bibitem{D8} S. P. Dutta, On modules of finite projective dimension, \emph{Nagoya Math J.} \textbf{219} (2015) 87--111.

\bibitem{D9} S. P. Dutta, On a consequence of the order ideal conjecture, \emph{JPAA} \textbf{219} no. 3 (2015) 482--487.

\bibitem{DG} S. P. Dutta, P. Griffith, Intersection multiplicities, the canonical element conjecture and the syzygy problem, \emph{Michigan Math. J.} \textbf{57} (2008) 227--247.

\bibitem{EG1} E. G. Evans, P. Griffith, The syzygy problem, \emph{Ann. of Math.} \textbf{114} no. 2 (1981) 323--333.

\bibitem{EG2} E. G. Evans, P. Griffith, Order ideals, in: Commutative Algebra, \emph{Math. Sci. Res. Inst. Publ.}, vol. 15, Springer (1989) 213--225.

\bibitem{EG3} E.G. Evans, P. Griffith, Syzygies, Lecture Note Series 106, London Mathematical Society, Cambridge University Press, 1985.

\bibitem{Go} S. Goto, On the associated graded rings of parameter ideals in Buchsbaum rings, \emph{J. Algebra} \textbf{85} no. 2 (1983) 490--534.

\bibitem{He} R. Heitman, The direct summand conjecture in dimension three, \emph{Ann. of Math.} \textbf{156} no. 2 (2002) 695--712.
\bibitem{H1} M. Hochster, Contracted ideals from integral extensions of regular rings, \emph{Nagoya Math. J.} \textbf{51} (1973) 25--43.

\bibitem{H2} M. Hochster, Topics in the Homological Theory of Modules Over Commutative Rings, \emph{CBMS Reg. Conf. Ser. Math.} vol. 24, Amer. Math. Soc., Providence, RI, 1975.

\bibitem{H3} M. Hochster, Canonical elements in local cohomology modules and the direct summand conjecture, \emph{J. Algebra} \textbf{84} (1983) 503--553.

\bibitem{K} J. Koh, Degree $p$ extensions of an unmixed regular ring of mixed characteristic $p$, \emph{J. Algebra} \textbf{99} (1986) 310--323.

\bibitem{O} T. Ohi, Direct summand conjecture and descent of flatness, \emph{Proc. Amer. Math. Soc.} \textbf{124} (1996) 1967--1968.

\bibitem{V} J. Valez, Splitting results in module-finite extension rings and Koh's conjecture, \emph{J. Algebra} \textbf{172} (1995) 454--469.

\end{thebibliography}
\end{document}